\documentclass[oneside,centertags,12pt]{amsart}
\usepackage{amsfonts}
\usepackage{amsthm}
\usepackage{amsmath}
\usepackage{amssymb}
\usepackage[all]{xy}
\usepackage[latin1]{inputenc}

\theoremstyle{plain}
\newtheorem{theorem}{Theorem}[section]
\newtheorem{corollary}[theorem]{Corollary}

\newtheorem{proposition}[theorem]{Proposition}
\theoremstyle{definition}

\theoremstyle{remark}

\newcommand{\ZZ}{\mathbb{Z}}

\SelectTips{cm}{}

\numberwithin{equation}{section}

\begin{document}

\title{Irrationality of the sum of a \boldmath$p$-adic series}
\author{S\'{\i}lvia Casacuberta}
\address{Aula Escola Europea, Mare de D\'eu de Lorda 34, 08034 Barcelona}
\email{silvia.casac@gmail.com}
\date{}

\thanks{}

\begin{abstract}
We prove that the sum of the series $\sum_{n=0}^{\infty}\, p^{v_p(n!)}$ is a $p$-adic irrational for all primes~$p$, where $v_p(n!)$ denotes the exponent of the highest power of $p$ dividing $n!$.
\end{abstract}

\maketitle

\section{Introduction}
Let $\ZZ_p$ denote the ring of $p$-adic integers.
One of the features of $p$-adic analysis is that a series $\sum_{n=0}^{\infty}\, x_n$ converges in $\ZZ_p$ if and only if $x_n\to 0$ in the $p$-adic metric \cite{Koblitz,Schikhof}. An intriguing special case is the factorial series $\sum_{n=0}^{\infty}\, n!$, whose sum in $\ZZ_p$ is conjectured to be irrational \cite{Dragovich,MS} for all primes~$p$, although this remains an open problem since Schikhof's remark in \cite[p.\,17]{Schikhof}.

The rate of convergence of $n!$ to zero in $\ZZ_p$ is determined by the highest power of $p$ dividing $n!$. According to Legendre's formula \cite{Legendre}, the exponent of this power is equal to
\begin{equation}
\label{legendre}
v_p(n!)=\frac{n-s_p(n)}{p-1}
\end{equation}
where $s_p(n)$ is the sum of the digits in the base $p$ expansion of~$n$.

In this note we prove that the series $\sum_{n=0}^{\infty}\, p^{v_p(n!)}$ converges to an irrational number $\alpha_p$ in $\ZZ_p$ for all~$p$. This is shown by demonstrating that the $p$-adic expansion of $\alpha_p$ is not periodic.

\bigskip

\noindent
{\bf Acknowledgements.} I am thankful to Dr.\ Xavier Taix\'es for his guidance and encouragement.

\section{Periodicity and rationality}
\label{section1}

We rely on the fact that the $p$-adic expansion of a $p$-adic number $\alpha$ is periodic if and only if $\alpha$ rational, that is, $\alpha=a/b$ where $a$ and $b$ are ordinary integers. The proof of this fact is similar to the proof of the analogous assertion concerning real numbers in terms of their decimal expansions \cite[p.\,106]{Neukirch}.

Let us denote by $s_p(n)$ the sum of all the digits in the base $p$ expansion of a positive integer~$n$. The set of numbers $s_p(n)$ where $n$ ranges from $kp$ to $(k+1)p-1$ will be called the \emph{$k$-th package}. Thus the cardinality of each package is equal to~$p$.

\newpage

\begin{proposition}
The $p$-adic valuation of $n!$ is constant in each package.
\end{proposition}

\begin{proof} 
This follows from \eqref{legendre}, since the numbers in each package are consecutive and there is only one multiple of~$p$ at the beginning of the package.
\end{proof}

In other words, $p^{v_p(n!)}$ is identical for all the numbers $n$ in any given package. 
Note also that, when adding the values of $p^{v_p(n!)}$, there is precisely one carry occurring in each package. For instance, when $p=3$, the first partial sums $\sum_{m=0}^n\, p^{v_p(m!)}$ grouped by packages are the following:
\[
1,2,10, \quad 20,100,110, \quad 210,1010,1110, \quad 11110, 21110,101110,\dots
\]

\begin{proposition}
\label{prop2}
If $n$ is the largest number in a package, then the number of digits of the $p$-adic expansion of $\sum_{m=0}^n\, p^{v_p(m!)}$ is equal to $v_p(n!)+2$.
\end{proposition}

\begin{proof}
This follows by induction over $k$ using the fact that there is one carry in each package.
\end{proof}

\begin{theorem}
The $p$-adic expansion of the sum of the series $\sum_{n=0}^{\infty}\, p^{v_p(n!)}$ in $\ZZ_p$ is not periodic.
\end{theorem}

\begin{proof}
We focus on what happens when $n$ is a power of $p$. Suppose that $n=p^r$ with $r\ge 1$ and compare the number of digits of a partial sum $\sum_{m=0}^{n-1}\, p^{v_p(m!)}$ with that of the next sum $\sum_{m=0}^n\, p^{v_p(m!)}$. Since $n=p^r$, we have $s_p(n)=1$ and $s_p(n-1)=r(p-1)$; hence
\[
v_p((n-1)!)=\frac{p^r-1}{p-1}-r=v_p(n!)-r.
\]

Consequently, the difference between the number of digits of $\sum_{m=0}^{n}\, p^{v_p(m!)}$ and that of the previous partial sum in the case when $n=p^r$ is precisely $r-1$, by Proposition~\ref{prop2}. This implies that the $p$-adic expansion of $\sum_{m=0}^n\, p^{v_p(m!)}$ starts with $1$ followed by a string of $r-2$ zeroes if $n=p^r$ with $r\ge 3$, and this string of zeroes is forever fixed in the subsequent partial sums. Since $r$ keeps increasing by $1$ each time a power of $p$ is encountered, the $p$-adic expansion of $\sum_{n=0}^{\infty}\, p^{v_p(n!)}$ cannot be periodic.
\end{proof}

\begin{corollary}
The sum of the series $\sum_{n=0}^{\infty}\, p^{v_p(n!)}$ is irrational in $\ZZ_p$ for all primes~$p$.
\end{corollary}

\end{document}